\documentclass[11pt,reqno]{amsart}
\usepackage{amsfonts,amsmath,amssymb}
\usepackage[cp1251]{inputenc}
\usepackage[T2A]{fontenc}
\newtheorem{thm}{Theorem}[section]
\newtheorem{proposition}[thm]{Proposition}
\newtheorem{corollary}[thm]{Corollary}
\newtheorem{lemma}[thm]{Lemma}
\newtheorem{definition}[thm]{Definition}

\newtheorem{remark}[thm]{Remark}
\usepackage{graphicx}
\usepackage{epstopdf}

\title[On rings  of supersymmetric polynomials]{On rings of supersymmetric polynomials}
\author{ A.N. Sergeev}\address{Department of Mathematics, Saratov State University, Astrakhanskaya 83, Saratov 410012, Russia and National Research University Higher School of Economics, Russian Federation.}
\email{SergeevAN@info.sgu.ru}

\begin{document}

\maketitle

\begin{abstract} We consider three types of rings of supersymmetric polynomials: polynomial ones $\Lambda_{m,n}$, partially polynomial $\Lambda_{m,n}^{+y}$ and Laurent supersymmetric rings $\Lambda_{m,n}^{\pm}$.  For  each type of rings we give their descriptions in terms of generators and relations.   As a corollary we get  for $n\ge m$ an isomorphism  $\Lambda_{m,n}^{+y}=\Lambda_{m,m}^{+y}\otimes\Lambda^{+y}_{0,n-m}$.  It is also true for polynomial  rings, but in this case the isomorphism does not  preserve the grading. For each type of rings  some natural basis consisting of Euler supercharacters is constructed.

 \end{abstract}

\tableofcontents

\section{Introduction}  

 Rings of symmetric polynomials $\Lambda_m$ play  an important role in many areas of mathematics (see \cite{Ma}). From the representation theory point of view ring   $\Lambda_m =\Bbb Z[x_1,\dots,x_m]^{S_m}$
 is the Grothendieck ring of polynomial  finite dimensional  representations  of the algebraic group $GL(m)$.  The ring $\Lambda_m$ has a natural linear  basis consisting of the classes of irreducible polynomial representations. Corresponding symmetric polynomials  are known as Schur polynomials. They can be expressed by the Weyl formula or  the Jacobi-Trudy formula.  It is also well known that the ring   $\Lambda_m$  is freely generated by elementary symmetric polynomials $e_1,\dots, e_m$ as well as   complete symmetric polynomials $h_1,\dots, h_m$ and  both families are algebraically independent. So the problem of describing this ring in terms of generators and relations is trivial in this case. 

 If we consider all finite dimensional representation of the algebraic group $GL(m)$ then the corresponding ring is the ring of symmetric Laurent polynomials  (see for example \cite{FH})
  $
\Lambda^{\pm}_{m}=\Bbb C[x_1^{\pm1},\dots,x_m^{\pm1}]^{S_m}
.$
 The  ring $\Lambda^{\pm}_{m}$ has not been studied  in such details  as the polynomial one. It has of course a natural liner basis consisting of characters of irreducible finite dimensional representations which can be given again by the Weyl formula.  There is also some analogue of  the Jacobi-Trudy formula, which is called composite Schur functions (\cite{King,Mo}). It is also well known that  $\Lambda^{\pm}_{m}=\Lambda_{m}[e_m^{-1}]$. And a natural question is: How to describe this ring in terms of complete symmetric functions? It turns out, that an answer to this question can be naturally given in terms of Euler characters $H_k(x_1,\dots,x_n),\,k\in \Bbb Z$
$$
H_k=E(k\varepsilon_1)=\frac{\{x_1^{k+m-1}x_2^{m-2}\dots x_m^0\}}{\{x_1^{m-1}x_2^{m-2}\dots x_m^0\}}
$$
where $\{\}$ means the alternation on symmetric group $S_m$. It is easy to check, that 
$H_{k}= h_{k}-h^{\infty}_k,\,k\in \Bbb Z$ where $h_k, h_k^{\infty}$ are coefficients  of Laurent series at zero and infinity of  the function $\frac{1}{\prod(1-x_it)}$. It looks like these polynomials play  the same role for Laurent symmetric polynomials  as polynomials $h_{k}$  for usual symmetric polynomials. 
For example Jacobi-Trudy identity in this case can be written in the usual form
$$
E_{\lambda}(x_1,\dots, x_m)=\det (H_{\lambda_i-i+j})_{1\le i,j\le m}
$$
Polynomials  $H_k$  also allow to  give a description of the algebra $\Lambda_m^{\pm}$ in terms of generators $h_1,h_1^*,\dots$ and relations between them.

A remarkable feature of this approach  is that it works in the same manner for rings of supersymmetric polynomials of all types as well.   And it turns out that  in supersymmetric case we need to consider function $\frac{\prod(1-y_jt)}{\prod(1-x_it)}$ and define $H_{k}= h_{k}-h^{\infty}_k,\,k\in \Bbb Z$ by the same formula as before. Then  relations for rings of supersymmetric polynomials (of all types) can be obtained by the same formulae as in $\Lambda^{\pm}_{m}$ case simply replacing number $m$ by the super-dimension $m-n$. We should say that the importance of considering Laurent decomposition at infinity was first observed in the paper \cite{KV}, and some of the relations for algebras supersymmetric polynomials $\Lambda_{m,n}$ were introduced in the paper \cite{KT}. 
 
 As it  was proved in \cite{SV}  the  ring $\Lambda_{m,n}^{\pm}$ is  the quotient of the Grothendieck ring of the category  finite dimensional representations of Lie supergroup $GL(m,n)$  by the relation $[M]=-[\Pi M]$ and it has a natural basis consisting of the classes irreducible finite dimensional representations. But there is no simple explicit formula for them \cite{Serga}, \cite{Brun}. So  instead of characters of irreducible modules we use Euler supercharacters following explicit formula by Serganova \cite{GS} (We should note that Euler supercharacters can be easily obtained from  the corresponding Euler characters). We also prove Jacobi-Trudy identity for Euler characters.
One of the main technical tool in this paper is the evaluation homomorphism  $\varphi:\Lambda^{\pm}_{m,n}\rightarrow \Lambda^{\pm}_{m-1,n-1}$ such that $\varphi(x_m)=\varphi(y_n)$. It has been shown recently \cite{HR}, that  this evaluation homomorphism  can be naturally induced by Duflo-Serganova functor. It would be interesting to give similar interpretation to other results of the present paper, for example Corollary \ref{prod2} and Corollary \ref{prod1} which state that  for $n\ge m$ we have an isomorphisms $\Lambda_{m,n}^{+y}=\Lambda_{m,m}^{+y}\otimes\Lambda^{+y}_{0,n-m}$ and $\Lambda_{m,n}=\Lambda_{m,m}\otimes\Lambda_{0,n-m}$.

\section{Laurent symmetric polynomials}

In this section we are going to generalise some of the facts about symmetric polynomials to the case of Laurent symmetric polynomials. The main result of this section is a description of the rings of Laurent polynomials in terms of generators and relations. This problem is trivial for rings of usual symmetric polynomials, but it is not trivial for the rings of Laurent symmetric polynomials. 

\begin{definition} Let us define 
$$
\Lambda^{\pm}_m=\Bbb Z[x_{1}^{\pm1},\dots,x_{m}^{\pm1}]^{S_{m}}
$$
 and we shall call this ring as the  ring of Laurent symmetric polynomials.
\end{definition}
\begin{definition} Let  $\lambda_1,\dots,\lambda_m$ be any sequence of integers. Let us define Euler character  $E_{\lambda}\in\Lambda^{\pm}$  by the following formula 
$$
E_{\lambda}(x)\Delta_m(x)=\{E_{\lambda}(x)x_1^{m-1}\dots,x_m^0\}=\{x_1^{\lambda_1+m-1}\dots x_m^{\lambda_m}\}
$$
where  $\Delta_m(x)=\prod_{i<j}(x_i-x_j)ra$ and $\{f(x)\}$ means alternation over group  $S_m$,
$$
\{f(x)\}=\sum_{\sigma\in S_m} \varepsilon(\sigma)\sigma(f),\,
$$
\end{definition}
\begin{thm}\label{basis}  Let $\lambda_{1}\ge\dots\ge\lambda_{m}$ be integers, then $E_{\lambda}(x)$ form a basis of the ring  $\Lambda^{\pm}_{m}$.
\end{thm}
\begin{proof} If $\lambda_{1}\ge\dots\ge\lambda_{m}$  then polynomials $\{x_1^{\lambda_1+m-1}\dots x_m^{\lambda_m}\}$ form a basis in the space of Laurent skew-symmetric polynomials. Multiplication by $\Delta_m$  gives an isomorphism between vector space of symmetric polynomials and vector space of skew-symmetric polynomials.
\end{proof}
\begin{definition}
For any integer  $k$ let us set  
$
H_k=E_{(k,0,\dots,0)}.
$

\end{definition}

Let us  also define polynomials $h_{k},\, h_{k}^{(\infty)}$ by equalities
$$
\frac{1}{\prod_{i=1}^m(1-x_{i}t)}=\sum_{k=0}^{\infty}h_{k}t^{k}\,\,
=\sum_{k=-\infty}^{-m}h_{k}^{(\infty)}t^{k}
$$
We as usually suppose that  $h^{*}_{k}=h_{k}=0$ if $k<0$, where $h^{*}_{k}=h_{k}(x_{1}^{-1},\dots,x_{m}^{-1})$.

It is easy to check that
$$
h_{k}^{(\infty)}=(-1)^m(x_{1}\dots x_{m})^{-1}h^{*}_{-m-k},
$$

\begin{lemma}\label{induc} We have the following equalities
     
$1)$    
$$
H_{k}=\begin{cases} h_{k},\,\, k\ge0\\
0,\,\, -m<k<0\\
-h^{(\infty)}_k\,\, k\le -m
\end{cases},\quad\quad\quad\text{or}\quad H_k=h_k-h_k^{(\infty)}
$$

$2)$     $H_{k}-x_{1}H_{k-1}=H_{k}(x_{2},\dots,x_{m})$
  
  $3)$ If $m=1$, then $H_{k}(x_1)-x_{1}H_{k-1}(x_1)=0$

\end{lemma}
\begin{proof} To prove the first equality let us calculate the following generating function 
$$
H_m^+(t)=\sum_{k=0}^{\infty}H_kt^k
$$
We have
$$
H_m^+(t)\Delta_m(x)=\sum_{k=0}^{\infty}\left\{x_1^kt^kx_1^{m-1},\dots,x_m^0\right\}=\sum_{k=0}^{\infty}\left\{\frac{1}{1-x_1t}x_1^{m-1},\dots,x_m^0\right\}
$$
$$
=\frac{x_1^{m-1}\Delta^{(1)}(x)}{1-x_1t}-\frac{x_2^{m-1}\Delta^{(2)}(x)}{1-x_2t}-\dots-\frac{x_m^{m-1}\Delta^{(m)}(x)}{1-x_mt}
$$
where
$$
\Delta^{(l)}(x)=\prod_{i,j\ne l, i<j}(x_i-x_j)
$$
 On the other hand we have the decomposition in  partial fractions 
 $$
\frac{1}{\prod_{i=1}^m(1-x_it)}=\frac{x_1^{m-1}}{\prod_{j\ne1}(x_1-x_j)}\frac{1}{1-x_1t}
$$
$$
+\frac{x_2^{m-1}}{\prod_{j\ne2}(x_2-x_j)}\frac{1}{1-x_2t}+\dots+\frac{x_m^{m-1}}{\prod_{j\ne m}(x_m-x_j)}\frac{1}{1-x_mt}
$$
Therefore
$$
H_m^+(t)=\frac{1}{\prod_{i=1}^m(1-x_it)}
$$
Now let us calculate generating function
$$
H_m^-(t)=\sum_{k=-1}^{-\infty}H_kt^k
$$
We have
$$
H_m^-(t)\Delta(x)=\sum_{k=1}^{\infty}\{x_1^{-k}t^{-k}x_1^{m-1},\dots,x_m^0\}=\sum_{k=0}^{\infty}\{\frac{x_1^{-1}t^{-1}}{1-x^{-1}_1t^{-1}}x_1^{m-1},\dots,x_m^0\}
$$
So $H_m^{-}(t)=-H_m^+(t)$.  This proves the first statement.

 Now let us prove the statement  $2)$. It is enough to prove it separately  for $h_{k}$ and $h_{k}^{(\infty)}$. 
 We have 
$
(1-x_{1}t)H^+_m(t)=H^+_{m-1}(t)
$
So the statement $2)$ is true for $h_{k}$. The case  $h_{k}^{(\infty)}$ can be proved in the same way.
 
 It it easy to check that  if $m=1$, then $H_k(x)=x^k,k\in \Bbb Z$ and the statement  $3)$ follows.
\end{proof}

\begin{thm} \label{alt} $1)$  Let $\lambda_1,\dots,\lambda_m$ be any sequence of integers, then

$$
E_{\lambda}(x_1,\dots, x_m)=\det (H_{\lambda_i-i+j})_{1\le i,j\le m}
$$

$2)$ If  $\lambda_1, \dots, \lambda_{m+1}$  any sequence of integers, then we have the following  equality in the ring $\Lambda_m^{\pm}$ 
$$
\det (H_{\lambda_i-i+j})_{1\le i,j\le m+1}=0
$$
\end{thm}
\begin{proof}    To prove the first statement we will use induction on $m$. If $m=1$, then  $E_k(x)=H_k(x)$ by definition and  the proof is trivial.  Let $m>1$. 
Let us denote by  $F_{\lambda}(x_1,\dots,x_m)$ the determinant 
$$
\left|\begin{array}{cccc}
   H_{\lambda_{1}}& H_{\lambda_{1}+1}& \ldots & H_{\lambda_{1}+m-1}\\
\vdots&\vdots&\ddots&\vdots\\
  H_{\lambda_{m}-m+1}& H_{\lambda_{m}-m+2}& \ldots & H_{\lambda_{m}}\\
  \end{array}\right|
$$
Then we have 
$$
F_{\lambda}(x_1,\dots,x_m)\Delta_m(x)=\left|\begin{array}{cccc}
   H_{\lambda_{1}}\Delta_m(x)& H_{\lambda_{1}+1}\Delta_m(x)& \ldots & H_{\lambda_{1}+m-1}\Delta_m(x)\\
\vdots&\vdots&\ddots&\vdots\\
  H_{\lambda_{m}-m+1}& H_{\lambda_{m}-m+2}& \ldots & H_{\lambda_{m}}\\
  \end{array}\right|
$$
 Since $H_k\Delta_m(x)=\{x_1^{k+m-1}x_2^{m-2}\dots x_m^0\}$  we get
$$
F_{\lambda}(x_1,\dots,x_m)\Delta_m(x)=\left\{\left|\begin{array}{cccc}
x_{1}^{\lambda_{1}}& x_{1}^{\lambda_{1}+1}& \ldots & x_{1}^{\lambda_{1}+m-1}\\
   H_{\lambda_{2}-1}& H_{\lambda_{2}}& \ldots & H_{\lambda_{2}+m-2}\\
\vdots&\vdots&\ddots&\vdots\\
  H_{\lambda_{m}-m+1}& H_{\lambda_{m}-m+2}& \ldots & H_{\lambda_{m}}\\
  \end{array}\right|x_{1}^{m-1}\dots x_{m}^0\right\}
$$
Now let us multiply every column (except the last one)  by $x_{1}$ and  subtract  the result from the following column. Then  we can apply  formula   $2)$ from Lemma \ref{induc}  and expand the determinant    along the first row. Then we get

\begin{equation}\label{ind1}
F_{\lambda}(x_1,\dots,x_m)\Delta_m(x)=\left\{F_{\mu}(x_2,\dots,x_m)x_{1}^{\lambda_{1}+m-1}x_{2}^{m-2}\dots x_{m}^0\right\}
\end{equation}
where $\mu=(\lambda_2,\dots,\lambda_m)$ and we can apply  inductive assumption.  

Let us prove the second statement also induction by $m$. If  $m=1$, then for any integers $\lambda_1, \lambda_2$ we have 
$$
\left|\begin{array}{cc}
H_{\lambda_{1}}& H_{\lambda_{1}+1}\\
   H_{\lambda_{2}-1}& H_{\lambda_{2}} \\
  \end{array}\right|=\left|\begin{array}{cc}
x^{\lambda_{1}}& x^{\lambda_{1}+1}\\
   x^{\lambda_{2}-1}& x^{\lambda_{2}} \\
  \end{array}\right|=0
  $$
If  $m>1$,  then we can use  the same arguments as before and induction.
\end{proof}
\begin{remark}\label{trudy} It is easy to check, that  if  $l\le m$, then
$$
\det (H_{\lambda_i-i+j})_{1\le i,j\le m}=\det (H_{\lambda_i-i+j})_{1\le i,j\le l}
$$
  where $\lambda_{l+1}=\dots=\lambda_{m}=0$.
\end{remark}
\begin{corollary}\label{generators} Polynomials $H_k,\, k\in\Bbb Z$ generate the ring  $\Lambda^{\pm}_m$.
\end{corollary}

There exists one more formula of Jacobi-Trudy type  for polynomials $E_{\lambda}$. It is called composite Schur functions \cite {King,Mo}.
\begin{lemma}\label{Com} Let  $\lambda=(\lambda_1,\dots,\lambda_m)$  be non-increasing sequence of integers. Let us write it in the form
$$
(\lambda_1,\dots,\lambda_m)=(\mu_1,\dots,\mu_r,0,\dots,0,-\nu_s,\dots,-\nu_1)
$$
where, $\mu,\,\nu$ are some partitions of the length  $r$ и $s$  correspondently. Then the following equality is valid 
$$
E_{\lambda}=\left|\begin{array}{cccc}
   h^*_{\nu_{s}}& h^*_{\nu_{s}-1}& \ldots & h^*_{\nu_{s}-s-r+1}\\
\vdots&\vdots&\ddots&\vdots\\
  h^*_{\nu_{1}+s-1}& h^*_{\nu_{1}+s-2}& \ldots & h^*_{\nu_{1}-r}\\
   h_{\mu_{1}-s}& h_{\mu_{1}-s+1}& \ldots & h_{\mu_{1}+r-1}\\
\vdots&\vdots&\ddots&\vdots\\
  h_{\mu_{r}-s-r+1}& h_{\mu_{r}-s-r+2}& \ldots & h_{\mu_{r}}\\
\end{array}\right|
  $$
\end{lemma}

\begin{proof}  It is easy to see that we can suppose that $r+s=m$,  considering patritions with zero parts. 
Since the determinant has size  $m\times m$, then by Lemma  \ref{induc} the following equalities are valid
$
h_i=H_i,\,\,h^*_j=H_{-m}^{-1}H_{-m-j}
$
for all elements of the determinant. So we can  bring $H_{-m}^{-s}$ in front of the determinant and take into account that  $H_{-m}^{-1}=(-1)^{m+1}x_1\dots x_m$. Then after reordering rows and taking sings into account  we get
$$
E_{\lambda}=(x_1\dots x_m)^s\left|\begin{array}{cccc}
H_{\mu_{1}-s}& H_{\mu_{1}-s+1}& \ldots & H_{\mu_{1}+r-1}\\
\vdots&\vdots&\ddots&\vdots\\
  H_{\mu_{r}-s-r+1}& H_{\mu_{r}-s-r+2}& \ldots & H_{\mu_{r}}\\
   H_{-m-\nu_{s}}& H_{-m-\nu_{s}+1}& \ldots & H_{-\nu_{s}-1}\\
\vdots&\vdots&\ddots&\vdots\\
  H_{-m-\nu_{1}-s+1}& H_{-m-\nu_{1}-s+2}& \ldots & H_{-\nu_{1}-s}\\
\end{array}\right|
  $$
But it it easy to see from the definition, that
$$
(x_1\dots x_m)E_{\lambda_1.\dots,\lambda_m}(x_1,\dots,x_m)=E_{\lambda_1+1.\dots,\lambda_m+1}(x_1,\dots,x_m)
$$
 and Lemma follows.
\end{proof}

Now we are going to describe the  ring $\Lambda_m^{\pm}$ in terms of  generators and relations.  First, we  prove  the following Lemma.
\begin{lemma}\label{identity} Let  $A$ be a matrix of the size  $n\times(n+1)$ and  
$$
A=\left(A_{1},A_{2},\dots A_{n+1}\right)
$$
its columns. Let   $A^{(l)}=A\setminus A_l$ be matrix of the size $n\times n$  obtaning from  $A$    by deleting the column $A_l$.  For any subset  $I\subset \{1,\dots,n\}$ define also  matrix $A(I)$ of the size $n\times n$  by  the formula
$$
r_i(A(I))=\begin{cases}r_i(A^{(n+1)}),\,i\in I\\
r_i(A^{(1)}),\,\, i\notin I
\end{cases}
$$
where $r_i(B)$  denotes  the  $i$-th row  of the matrix $B$. 

Then the following equality is true
$$
\det A^{(l)}=\sum_{I\subset \{1,\dots, n\}}\det A(I)
$$
and sum is taken over all subset cardinality   $l-1$.
\end{lemma}
\begin{proof}
Let us consider matrix $A^{(1)}+aA^{(n+1)}$, where  $a$ is indeterminate and decompose determinant in two ways: using rows and columns. 
 If we use multilinear property  of  the determinant in rows, then we get  
$$
\det\left(A^{(1)}+aA^{(n+1)}\right)=\sum_{l=0}^na^l\sum_{ Card(I)=l}\det A_{I}
$$
where $Card(I)$ denotes the number of elements in $I$.
 If we use multilinear property of the  determinant in columns then we get 
 $$
\det\left(A^{(1)}+aA^{(n+1)}\right)=\sum_{l=1}^{n+1}a^{l-1}\det A^{(l)}
$$
and proof follows.
\end{proof}
\begin{definition}\label{Re} Let $z=(z_i),i\in\Bbb Z$ be an infinite sequence of variables and $I=(i_1,\dots i_p)\in{\Bbb Z}^p$ finite sequence of integers. Set 
$$
R_I(z)=\det(z_{i_{\alpha}+\beta-1})_{1\le,\alpha,\beta\le p}
$$
We will   call $p$ as the length of $I$ and denote it by $ \hat l(I)$.
\end{definition}
Now we are ready to describe the  ring  $\Lambda_m^{\pm}$ in terms of the generators and relations. 
\begin{thm}\label{relations1} Ring $\Lambda_m^{\pm}$ is isomorphic any of the following rings:

$1)$   $U^+_{m,0}$ which is generated by  $u_i,\hat v_i,\, i\in \Bbb N$ with relations
$$
R_{I}(w)=0,\,\,\, \text{for any}\,\, I\in \Bbb Z^{m+1}
$$ 
where $w_i=u_i-\hat v_{-m-i},\, i\in \Bbb Z$ and we suppose that $u_0=1$ и $u_i=\hat v_i=0,\,i<0$.

$2)$   $U^{\pm}_{m,0}$ which is  generated by  $t, u_i,v_i,\, i\in \Bbb N$ with relations
$$
R_{I}(w)=0,\,\,\, \text{for any}\,\, I\in \Bbb Z^{m+1}
$$ 
 where $w_i=u_i-tv_{-m-i},\, i\in \Bbb Z$ and we assume that  $u_0=v_0=1$ and $u_i=v_i=0,\,i<0$.

\end{thm}
\begin{proof} Let us prove the first statement. Let $S=\{u_1,\dots, \hat v_0,\dots\}$  be the set of generators of the ring $U_{m,0}$. Consider a map  
$$
\varphi: S \longrightarrow \Lambda_m^{\pm},\,\, \varphi(u_i)=h_i,\,i\ge1,\,\,\varphi(\hat v_i)=(-1)^m(x_1\dots x_m)^{-1}h^*_{i},\,i\ge0
$$
Then it is easy to check that $\varphi(w_i)=H_i,\,i\in \Bbb Z$.  Therefore the map $\varphi$ can be extended as the homomorphism on the   whole  ring $\varphi :U_{m,0}\rightarrow \Lambda_m^{\pm}$.  Now let us prove that this map  is an isomorphism. For that we are going to construct some family of linear generators of the ring $\Lambda_m^{\pm}$. Namely  we will prove that ring  $U_{m,0}$ is a linear span of the elements $R_{I}(w)$ where $I\in \Bbb Z^{m}$.  It is enough to prove that any product  $w_{j_1}\dots w_{j_r}$ for any $r\ge 0$ can be expressed as a linear combination  some of $R_I(w)$ with $\hat l(I)=m$. We will use induction on $r$. If  $r=0$ then the product is equal to $1$ and the following equality is easy to verify
$
1= R_{0,1,\dots,1-m}(w)
$.
So we can assume that  $r>1$ and we  will  prove that product $R_{i_1,\dots, i_m}(w)w_j$ 
 where $j\notin \{0,-1,\dots, 1-m\}$ is a linear combination some of $R_{j_1,\dots,j_m}(w)$.
 
 There are two cases  $j>0$ and  $j<1-m$.  Consider the first case. We are going to use induction on $j$. If $j=1$, then by definition of the ring  $U_{m,0}$ we have  $R_{\tilde I}(w)=0$, where $\tilde I=(i_1,i_2,\dots,i_m,1-m)$.
Expanding the determinant along its last row we get
$$
\sum_{i\in\{1-m,\dots,0,1\}}(-1)^{i+1}R_{\tilde I}^{(i)}(w) w_{i}=0
$$

where $R_{\tilde I}^{(i)}(w)$ means the  determinant obtaining from  $R_{\tilde I}(w)$  by deleting row and column containing element  $w_i$. Therefore 
$$
R_I(w)w_1+\sum_{i\in\{1-m,\dots,0\}}(-1)^{i+1}R_{\tilde I}^{(i)}(w) w_{i}=0
$$
So $R_I(w)w_1=R_{\tilde I}^{(0)}(w)$. But if we apply  Lemma \ref{identity}  to the matrix
$$
A= \left(\begin{array}{cccc}
   w_{i_{1}}& w_{i_{1}+1}& \ldots & w_{i_{1}+m}\\
\vdots&\vdots&\ddots&\vdots\\
  w_{i_{m}}& w_{i_{l}+1}& \ldots & w_{i_{l}+m}\\
  \end{array}\right)
$$
then we see that $R_{\tilde I}^{(0)}(w)$ can be expressed as a linear combination  some of $R_I(w),\,$ with $\hat l (I)=m$.
 If  $j>1$ then it is enough to consider relation $R_{\tilde I}(w)=0$, where $\tilde I=(i_1,\dots,i_m,j-m)$  and to use inductive assumption. If $j<1-m$, then we can use the same arguments applying to relation $R_{\tilde I}(w)=0$, where $\tilde I=(i_1,\dots,i_m,j)$.  Thereby we proved that $U_{m,0}$  is  a linear span    of $R(i_1,\dots,i_m)$ with $\hat l( I)=m$.  

Therefore by Theorem  \ref{alt}
 $$
 \varphi(R_{i_1,\dots,i_m}(w))=E_{i_1,i_2+1\dots,i_m+m-1}(x_1,\dots,x_m)
 $$
 and by  Theorem   \ref{basis}  elements $\varphi(R_{i_1,\dots,i_m}(w)),\,i_1>i_2>\dots>i_m$ form a basis  of the ring $\Lambda_m^{\pm}$. Therefore homomorphism  $\varphi$ is an isomorphism and we  have proved the first statement.

 Now let us prove the second statement. We shall prove that    $U_{m,0}$ and $U^{\pm}_{m,0}$ are isomorphic. From the definition of these rings there exists a homomorphism 
 $$
 \varphi : U_{m,0}\longrightarrow U^{\pm}_{m,0},\quad \varphi(u_i)=u_i,\,i\ge1,\,\varphi(\hat v_i)=tv_i,\,i\ge0
 $$
 Let us construct  an inverse homomorphism  $\psi : U^{\pm}_{m,0}\rightarrow U_{m,0}$. First let us prove  that  element  $\hat v_0$  is invertible in the ring $U_{m,0}$. 
 Consider relation 
$$
R_{0,1,\dots,m}(w)=\left|\begin{array}{cccc}
   w_{0}& w_{1}& \ldots & w_{m}\\
\vdots&\vdots&\ddots&\vdots\\
  w_{-m}& w_{-m+1}& \ldots & w_{0}\\
  \end{array}\right|=0
$$
Since  $w_0=1,w_{-1}=\dots=w_{1-m}=0$, we have an equality
$$
(-1)^mw_0\left|\begin{array}{cccc}
   w_{1}& w_{2}& \ldots & w_{m}\\
\vdots&\vdots&\ddots&\vdots\\
  w_{1-m}& w_{2-m}& \ldots & w_{1}\\
  \end{array}\right|=-1
$$
Therefore  $\hat v_0=-w_m$  is invertible and we can define homomorphism  $\psi$
$$
\psi : U^{\pm}_{m,0}\longrightarrow U_{m,0}, \quad \psi(u_i)=u_i,\,i\ge1,\,\psi(t)=\hat v_0,\,\psi( v_i)=\hat v_0^{-1}\hat v_i,\,i\ge1
$$
It is easy to check, that  $\varphi,\psi$ are mutually inverse homomorphisms. This proves the second statement and the Theorem.
\end{proof}
\begin{remark}  In the previous theorem we used two slightly different ways to define the ring of Laurent symmetric polynomials by means of generators and relations. We shall see later that a natural generalisation of the first way gives description of the ring of partially polynomial supersymmetric polynomials and a natural generalisation of the second way  gives a description of the ring Laurent supersymmetric polynomials.
\end{remark}

\section{Supersymmetric partially polynomial  and  polynomial  rings}

\begin{definition}  The following ring
$$
\Lambda^{+y}_{m,n}=\{f\in\Bbb Z[x^{\pm1}_{1},\dots,x^{\pm1}_{m},y_{1},\dots,y_{n}]\mid x_{i}\frac{\partial f}{\partial x_{i}}+y_{j}\frac{\partial f}{\partial y_{j}}\in (x_{i}-y_j)\}
$$
will be called the ring of partially polynomial (in $y$-s) supersymmetric polynomials.
\end{definition}
\begin{definition}  The ring 
$$
\Lambda_{m,n}=\{f\in\Bbb Z[x_{1},\dots,x_{m},y_{1},\dots,y_{n}]\mid x_{i}\frac{\partial f}{\partial x_{i}}+y_{j}\frac{\partial f}{\partial y_{j}}\in (x_{i}-y_j)\}
$$
will be called  the ring of super symmetric polynomials.
\end{definition} 
In this section we describe the rings $\Lambda^{+y}_{m,n}$ and $\Lambda_{m,n}$ in term of the generators and relations.

Let us define  $h_{k},\, h_{k}^{(\infty)}$ by means of expansion at zero and at infinity the following rational function
$$
\frac{\prod_{j=1}^n(1-y_{j}t)}{\prod_{i=1}^m(1-x_{i}t)}=\sum_{k=0}^{\infty}h_{k}t^{k}\,\,
=\sum_{k=-\infty}^{n-m}h_{k}^{(\infty)}t^{k}
$$
It is easy to see that 
$$
h_{k}^{(\infty)}=(-1)^{n-m}\frac{y_{1}\dots y_{n}}{x_{1}\dots x_{m}}h^{*}_{n-m-k},
$$
where $h^{*}_{k}=h_{k}(x_{1}^{-1},\dots,x_{m}^{-1},y_{1}^{-1},\dots,y_{n}^{-1})$. We also assume that  $h^{*}_{k}=h_{k}=0$, if $k<0$.
\begin{definition} For  $k\in\Bbb Z$ set 
\begin{equation}\label{hbig}
H_{k}=h_{k}- h_{k}^{(\infty)}=h_{k}-(-1)^{n-m}\frac{y_{1}\dots y_{n}}{x_{1}\dots x_{m}}h^{*}_{n-m-k}
\end{equation}

\end{definition}
\begin{remark} Previous formulae can be rewritten in the form
$$
H_{k}=\begin{cases} h_{k},\,\, k>n-m\\
h_k- h^{(\infty)}_{k}\,\,\, 0\le k\le n-m\\
-h^{(\infty)}_{k}\,\, k<0
\end{cases}
$$
\end{remark}
\begin{lemma}\label{alt1} The following equalities are valid 

$1)$
$$
H_{k}(x,y)=\sum_{j=0}^n(-1)^je_{j}(y)H_{k-j}(x)
$$

$2)$
$$
H_{k}(x)\Delta(x)\Delta(y)=\left\{\prod_{j=1}^n\left(1-\frac{y_{j}}{x_{1}}\right)x_{1}^kx_{1}^{m-1}x_{2}^{m-2}\dots x_{m}^0y_{1}^{n-1}\dots y_{n}^{0}\right\}
$$
and   $\{ f(x,y)\}$ means alternation over the group $S_{m}\times S_{n}$
$$
\{f(x,y)\}=\sum_{(\sigma,\tau)\in S_m\times S_n}\varepsilon(\sigma)\varepsilon(\tau)f(\sigma x,\tau y)
$$

$3)$
$$
H_{k}(x,y)-x_{1}H_{k-1}(x,y)=H_{k}(x_{2},\dots,x_{m},y)
$$

$4)$ if $m=1$, then for any integer  $k$ 

$$
H_{k}(x_1,y)-x_{1}H_{k-1}(x_1,y)=0
$$

$5)$ For any sequence of integers $\lambda_1,\dots,\lambda_m$ the following equality is true
$$
\det(H_{\lambda_i-i+j})_{1\le i,j\le m}=\prod_{i=1}^m\prod_{j=1}^n\left(1-\frac{y_j}{x_i}\right)E_{\lambda}(x_1,\dots,x_m)
  $$

$6)$ For any sequence of integers $\lambda_1,\dots,\lambda_{m+1}$ we have the following equality
$$
\det(H_{\lambda_i-i+j})_{1\le i,j\le m+1}=0
$$

\end{lemma}

\begin{proof}
Let us prove the first statement. It is enough to prove it separately  for  $h_k$ и $h_k^{(\infty)}.$ In the case of  $h_{k}$ it follows from the equality
  $$
 \sum_{i=0}^{\infty}h_{k}(x,y)t^k=\left(\sum_{j=0}^n(-1)^je_{j}(y)t^j\right) \sum_{i=0}^{\infty}h_{k}(x)t^k
 $$
 A proof for  $h_k^{(\infty)}$ is similar.
 
 The second statement follows from the first one and the definition of $H_k(x)$. 
 
 The third  and the forth  statements also  follow from the first one and Lemma  \ref{induc}. 
 
 Statements  $5),6)$ can be proved in the same manner as in Theorem   \ref{alt}, but instead of the definition   $H_k(x)$  we need to use the  statement  $2)$ from Lemma \ref{alt1}.
 \end{proof}
  
 In order to describe the algebra  $\Lambda^{+y}_{m,n}$ in terms of generators and relations we need to construct a linear basis in this algebra.
\begin{definition} Let  $I=(i_1,\dots,i_p)$ be a sequence of integers and  $J=(j_1,\dots,j_q)$ be a sequence of nonnegative integers. Set 
$$
H(I,J)= \left|\begin{array}{cccc}
   H_{i_{1}}& H_{i_{1}+1}& \ldots & H_{i_{1}+p-1}\\
\vdots&\vdots&\ddots&\vdots\\
  H_{i_{p}}& H_{i_{p}}& \ldots & H_{i_{p}+p-1}\\
  \end{array}\right| h_1^{j_1}h_2^{j_2}\dots h_q^{j_q}=R_{I}(H)h^J
$$
\end{definition}

Let us also denote by  $X^+(m,n)$ the set of pairs   of the sequences $(I,J)$   such that  $I$ strictly decreasing  sequence of integers, $J$ any sequence of nonnegative integers and   
$$
 \hat l(I)\le m,\quad \hat l(J)\le n,\quad \hat l( I)-\hat l (J)=m-n
$$

where as before  the  equality $\hat l(I)=p$ means that $I\in\Bbb Z^p$.
\begin{remark} If $\hat l(I)=0$, then we assume  that $I=\emptyset$ and $R_I(H)=1$ if $\hat l(J)=0$, then we assume that $h^J=1$.
\end{remark}

\begin{thm} \label{basis1}  Elements   $H(I,J),\,(I,J)\in X^{+}(m,n)$  form a linear basis of the ring   $\Lambda^{+y}_{m,n}$.
\end{thm}
\begin{proof}  Let use induction on $mn$. Let $mn=0$. If $n=0$, then the statement follows from Theorem \ref{alt}.  If $m=0$,  the the statement follows from the main theorem of symmetric functions. Let $mn>0$. L Consider a natural homomorphism
$$
\varphi_{m,n} : \Lambda^{+y}_{m,n}\longrightarrow \Lambda^{+y}_{m-1,n-1}
$$
such that $\varphi_{m,n}(x_m)=\varphi_{n,m}(y_n)=t,$ and   it acts identically on all other variables.  It is clear that  
$$
\varphi(H_i)=H_i,\,\,\,\varphi(h_j)=h_j
$$
 From the inductive assumption it follows that this homomorphism is surjection. Therefore it is enough to prove that the kernel of this homomorphism has a basis consisting of $H(I,J)$ such, that  $\hat l( I)=m,\,\hat l(J)=n$. 
It is easy to check that the following family forms a basis of the kernel  
$$
\prod_{i=1}^m\prod_{j=1}^n\left(1-\frac{y_j}{x_i}\right)E_{\lambda}(x)e_1(y)^{j_1}\dots e_n(y)^{j_n}
$$

where $\lambda_1\ge\lambda_2\ge\dots\ge\lambda_m$ non-increasing sequence of integers and  $j_1,\dots,j_n$ - any sequence of nonnegative integers and   $e_1,\dots,e_n$ elementary symmetric polynomials. Further we have 
$$
h_{1}(x,y)=h_{1}(x)-e_{1}(y)
$$
$$
h_{2}(x,y)=h_{2}(x)-h_1(x)e_1(y)+e_2(y)
$$
$$
\vdots\quad\quad\quad\quad\quad\vdots\quad\quad\quad\quad\quad\vdots\quad\quad\quad\quad\quad\vdots
$$
$$
 h_{n}(x,y)=h_{n}(x)-h_{n-1}(x)e_1(y)+\dots+(-1)^ne_{n}(y)
$$
 So we see that   $h_{1}(x,y),\,h_{2}(x,y),\dots,h_{n}(x,y)$ can be expressed by low-triangular   matrix in terms of   $e_1(y),\,\sigma_2(y),\dots,e_{n}(y)$ with units (up to sign) on the main diagonal.  Therefore there exists an automorphism   $\sigma$ of the algebra   
 $
 \Bbb C[x_1^{\pm1},\dots,x_m^{\pm1},y_1,\dots,y_n] ^{S_m\times S_m}$ such that
$$ 
  \sigma(e_i(y))=h_{i}(x,y),\,i=1,\dots,n,\quad \sigma (e_i(x))=e_i(x),\, i=1,\dots,m.
 $$
 Therefore by Lemma \ref{alt1} we see that
$$
R_{\lambda_{1},\lambda_2+1,\dots,\lambda_{m}+m-1}h_{1}^{j_{1}}h_{2}^{j_2}\dots h_{n}^{j_n}
$$
also form a basis of the kernel. And the Theorem follows from the inductive assumption.

\end{proof}

\begin{definition}   Let $m,n$ be two nonnegative integers. Let also  $u_1,u_2,\dots$ and  $v_0,v_1,\dots$ be two infinite sets. We will assume that  $u_0=1$ and $u_i=0,\, v_i=0$ for  $i<0$. Set $w_i=u_i-v_{-i-m+n},\, i\in\Bbb Z$ and  denote by  $ U^+_{m,n}$  the ring generated by  $u_1,u_2,\dots,$ and  $\, v_0,v_1,\dots,$ with relations 
$$
R_{i_1,\dots,i_{m+1}}(w)= 0,\,\,\, \text{for any} \,\,\, (i_1,\dots,i_{m+1})\in\Bbb Z^{m+1}
$$

\end{definition}

Now we want to construct some set of linear generator of the algebra $ U_{m,n}$.

\begin{definition} Let   $I=(i_1,\dots,i_p)$ be a  sequence of integers and  $J=(j_1,\dots,j_q)$ be a sequence of nonnegative integers. Set
$$
R(I,J)= R_{I}(w) u_1^{j_1}u_2^{j_2}\dots u_q^{j_q}
$$
\end{definition}

\begin{thm}\label{right} Elements   $R(I,J)$, such that $(I,J)\in X^{+}(m,n)$ linearly generate the ring $U^+_{m,n}$.

\end{thm}

\begin{proof}

 We will use induction on $mn$. Let  $mn=0$. Then, either $m=0$, or $n=0$. In the first case  $\hat l(I)=0,  \hat l( J)=n$, Therefore the relations become  $R_i(w)=0,i\in \Bbb Z$ and  they are equivalent  to the relations $u_i=0, i>n$ and $ u_i-v_{n-i}=0,\,0\le i\le n$ and $v_j=0,\,j>n$. Therefore  $U^+(0,n)$ is generated algebraically  by $u_1,\dots,u_n$  and in this case Theorem is true.

In the second case  $\hat l(I)=m,\,\hat l(J)=0$ and we need to show that  $U^+(m,0)$  is a linear span  of the elements  $R_I(w),\,\hat l(I)=m$.  But this follows from the Theorem \ref{relations1}. 

Let us now suppose that $mn>0$. We have  $m-n=(m-1)-(n-1)$, therefore from the definition of the rings  $U^+_{m,n}$ it follows that there exists  a homomorphism 
$$
\psi_{m,n} : U^+_{n,m}\rightarrow U^+_{m-1,n-1}
$$
which sends generators to generators.  Again from the definition of the rings  $U^+_{n,m}$  it follows that the kernel of this homomorphism is the ideal generated by  $R_I(w)$ with $\hat l( I)=m$. Therefore it is enough to prove that for  $j\ne1,\dots, n$ the product  
$R_I(w)u_j$ can be expressed as a linear combination  some  of $R(\tilde I,\tilde J)$. And we need to prove also that   product  $R_Iv_j$ for any  $j$ is a linear combination  some  of  $R(\tilde I,\tilde J)$ as well.  

Let us consider the first case. We can assume that  $j\ge n$  and we will use induction on  $j-n$. If  $j=n$, then our statement is clear. Let  $j>n$. Consider  relation
$R_{i_1,\dots,i_m,j-m}(w)=0$. If we expand the determinant along the last row then we get
$$
R_I(w)w_j+\sum_{i\in\{j-m,\dots,j-1\}}(-1)^{i+1}R_{\tilde I}^{(i)}(w) w_{i}=0
$$
Since  $i\in\{j-m,\dots,j-1\}$ we have  $i\ge j-m >n-m$, so $w_i=u_i$ (it would be zero, if $i<0$) and by induction  $R_I(w)w_j$ is a linear combination some of $R(I,J)$.
Let us prove now that  $R_{I}v_j$ with $\hat l( I) =m$ is a linear combination some of  $R(\tilde I,\tilde J)$ for  $j\le 0$ using induction on $j$. If  $j=0$, then  $w_{n-m}=u_{n-m}-v_0$. Therefore we can replace  $v_0$ on $w_{n-m}$ and we can consider relation $R_{i_1,\dots,i_m,n-m}=0$.  Using the same arguments as before we get necessary statement for  $j=0$. If $j>0$, then   $w_{n-m-j}=u_{n-m-j}-v_j$ and we can replace  $v_j$ by  $w_{n-m-j}$ and we can consider a relation  $R_{i_1,\dots,i_m,n-m-j}=0$ and use inductive assumption.  Theorem is proved.

\end{proof}
\begin{corollary}\label{iso} Rings  $\Lambda^{+y}_{m,n}$ and $U^+_{m,n}$ are  isomorphic.
\end{corollary}
\begin{proof} By Lemma   \ref{alt1}
$$
\det(H_{\lambda_i-i+j})_{1\le i,j\le m+1}=0
$$
 for any sequence  of integers  $\lambda_1,\dots,\lambda_{m+1}$.  Therefore from the defining relations of the ring $U^+_{m,n}$ it follows that there exists homomorphism such, that
$$
\varphi : U_{m,n}\longrightarrow \Lambda^{+y}_{m,n},\,\, \varphi(u_i)=h_i,\,\, \varphi(v_i)=(-1)^{n-m}\frac{y_1\dots y_n}{x_1\dots x_m}h^*_{i},\,\,i\ge 1
$$
This homomorphism sends the family of linear generators of the algebra  $U^+_{m,n}$ to a basis of the algebra  $\Lambda^{+y}_{m,n}$. Therefore it is an isomorphism.
\end{proof}

Let us consider the ring  $\Lambda_{m,n}$  of supersymmetric polynomials. It is a subring in $\Lambda_{m,n}^{+y}$. We also want to describe it in terms of generators and relations. Let us denote by $\Bbb Z_{>a}$ the set of integers which are strictly grater then $a$.

\begin{corollary}  $\Lambda_{m,n}$ is isomorphic to the ring  $U_{m,n}$ which is generated by $ u_1, u_2,\dots$ subject to relations 
$$
R_{I}(u)=0,\, \text{for any}\,\, I=(i_1,\dots,i_{m+1})\in(\Bbb Z_{>n-m})^{m+1}
 $$
and we assume that   $u_0=1$ и $u_i=0,\,i<0$. 
\end{corollary}
\begin{proof} It is well  known that ring  $\Lambda_{m,n}$ is generated by $h_1,h_2,\dots$ ( see for example \cite{Ma}). Therefore according to the previous Theorem the ring $\Lambda_{m,n}$ is isomorphic  to the subring  $U$ in  $U_{m,n}$ generated by $u_1, u_2,\dots$. From the definition $w_i$ it follows that if $i>n-m$ then $w_i=u_i$.  Therefore elements  $R(I,J),\,l(I)=m,\, I\in(\Bbb Z_{>n-m})^{m}$ (we call  such elements admissible)  belong to  $U$. 

 Let us consider relation  $R_{I}(u)=0,\hat l( I)=m+1$ , where $I$ is admissible (we  call such relations also admissible).  Then by the same arguments as in the proof of the Theorem   \ref{right}, it can be shown that  admissible elements linearly generate  $U$ by using only admissible   relations. But according to the corollary \ref{iso}  admissible elements  $R(I,J)$ are linearly independent. This  proves corollary.
\end{proof}

\begin{corollary}\label{prod2} Let  $n\ge m$, then the ring    $\Lambda^{+y}_{m,n}$(as the graded one) is isomorphic to the ring $\Lambda^{+y}_{m,m}\otimes \Lambda_{n-m}$.
\end{corollary}
\begin{proof} 
Let us consider  a map $\varphi : \Lambda^+_{m,m}\longrightarrow \Lambda^+_{m,n},\,\,$   
$$ 
 \varphi(h_i)=h_{n-m+i},\,\,\,\varphi(\frac{y_1\dots y_m}{x_1\dots x_m}h_i^{*})=(-1)^{n-m}\frac{y_1\dots y_n}{x_1\dots x_m}h_i^*,\,i\ge1
$$  
It follows from the defining relations of the ring $\Lambda^+_{m,m}$ that the map $\varphi$ can be extended to the  homomorphism  of the rings
$
\varphi : \Lambda_{m,m}\longrightarrow \Lambda_{m,n}
$
It is easy to check that    $\varphi(H_i)=H_{n-m+i}$, and  $\varphi(H(I,J))=H(I+n-m,J+n-m)$, where $I+a$ means  the sequence $(i_1+a,\dots,i_p+a)$.  So the map $\varphi$  sends the basis of the ring  $\Lambda^+_{m,m}$  to  a subset of the basis of the ring $\Lambda_{m,n}$. And $\varphi$ is injective when restricted to the  basis. Therefore  $\varphi$ is injective as a homomorphism of the rings. 

Further there exists  a homomorphism  $\psi$ such that 
$$
\psi :\Lambda_{n-m}\longrightarrow \Lambda_{m,n},\,\, \psi(h_i)=h_i,\,\,i=1,\dots, n-m
$$
So we have a homomorphism  
$$
\varphi\otimes\psi  : \Lambda^{+y}_{m,m}\otimes \Lambda_{m,n} \longrightarrow  \Lambda_{m,n}
$$
and  it is easy to see that homomorphism $\varphi\otimes\psi$  sends  bijectively tensor product of the bases in rings  $\Lambda_{m,m}$  and  $\Lambda_{n-m}$ to the basis of the ring  $\Lambda_{m,n}$.  Therefore this is  an isomorphism. But it does not preserves the grading. In order to construct a homomorphism preserving the grading let us consider a composition  $\varphi\otimes\psi\circ \delta$, where $\delta :  \Lambda^{+y}_{m,m}\longrightarrow \Lambda^{+y}_{m,m}$ is an automorphism $\delta(H_i)=H_{i-n+m}.$

\end{proof}

\begin{corollary}\label{prod1} Let  $n\ge m$. Then there exists an isomorphism  of rings  $$\Lambda_{m,m}\otimes\Lambda_{n-m}=\Lambda_{m,n}$$.
\end{corollary}
\begin{proof}  Let us consider the previous isomorphism  $\varphi\otimes\psi$ and restrict it to the subring   $\Lambda_{m,m}\otimes\Lambda_{n-m}$. It is clear that its image is  $\Lambda_{m,n}$.
\end{proof}

\begin{remark} It it easy to see that the isomorphism from the previous corollary does not preserves the grading.
\end{remark}

\section{Laurent supersymmetric polynomials}
\begin{definition}  The following ring
$$
\Lambda^{\pm}_{m,n}=\{f\in \Bbb Z[x_1^{\pm1},\dots,x_m^{\pm1},y_1^{\pm1},\dots,y_n^{\pm1}]\mid x_i\frac{\partial f}{\partial x_i}+y_j\frac{\partial f}{\partial y_j}\in (x_i-y_j)\}
$$
 will be called the ring of Laurent supersymmetric polynomials.
\end{definition}
 We are going to describe it in term of generators and relations. First we are going to construct some natural basis in this ring.

\begin{definition}\label{ij} Let  $I=(i_1,\dots,i_p)$  be a sequence of integers and  $J=(j_1,\dots,j_q)$ be a sequence of integers such that the first $q-1$ elements are nonnegative and the last one is any integer. Let us set 
$$
H(I,J)= R_I(H) h_1^{j_1}h_2^{j_2}\dots h_{q-1}^{j_{q-1}}\Delta^{j_q}
$$
where  $\Delta=\frac{y_1\dots y_n}{x_1\dots x_m}$.
\end{definition}

Let us denote by $X^{\pm}(m,n)$ the set of pairs $(I,J)$ with the same properties as in Definition \ref{ij}, but sequence $I$ is strictly decreasing  and such that
$$
\hat l( I)\le m,\,\,\hat l(J)\le n,\, \hat l (I)-\hat l(J)=m-n
$$

\begin{thm} Let  $(I,J)\in X^{\pm}(m,n)$, then elements    $H(I,J)$  form a linear basis of  the ring $\Lambda^{\pm}_{m,n}$.
\end{thm}
\begin{proof}  Let us use induction on  $mn$. Let $mn=0$. If  $n=0$, then $ \hat l(I)=m$  and the statement follows from Theorem \ref{alt}.  If $m=0$,  the statement follows form the fact that the set of element of the type  $e_1^{j_1}\dots e_{n-1}^{j_{n-1}}e_n^{j_n}$ form a linear basis in the ring of Laurent symmetric polynomials $\Bbb Z[y^{\pm1}_1,\dots,y_n^{\pm1}]^{S_n}$ . 

Let  $mn>0$. Consider a natural homomorphism
$$
\varphi_{m,n} : \Lambda^{\pm}_{m,n}\longrightarrow \Lambda^{\pm}_{m-1,n-1}
$$
such that  $\varphi_{m,n}(x_m)=\varphi_{n,m}(y_n)=t,$  and it acts identically on all other variables. It is clear that  
$$
\varphi(H_i)=H_i,\,i\in \Bbb Z, \quad\varphi(h_j)=h_j,\quad\varphi(h^*_j)=h^*_j,\,j\in \Bbb Z_{\ge 0}\quad \varphi(\Delta)=\Delta
$$
From the inductive assumption it follows that this homomorphism is surjection and it is enough to prove that the kernel of this homomorphism has  a basis consisting of  $H(I,J)$ such that   $\hat l( I)=m,\,\hat l(J)=n$. 
It is easy to check that that the following family of elements  forms a basis of the kernel
$$
\prod_{i=1}^m\prod_{j=1}^n\left(1-\frac{y_j}{x_i}\right)E_{\lambda}(x)e_1(y)^{j_1}\dots e_n(y)^{j_n}
$$

where $\lambda_1\ge\lambda_2\ge\dots\ge\lambda_m$ is a non-increasing sequence of integers and  $j_1,\dots,j_{n-1}$ is   any sequence of nonnegative integers, $j_n$ is any integer and   $e_1,\dots,e_n$ elementary symmetric polynomials.  Further we have 
$$
h_{1}(x,y)=h_{1}(x)-e_{1}(y)
$$
$$
h_{2}(x,y)=h_{2}(x)-h_1(x)e_1(y)+e_2(y)
$$
$$
\vdots\quad\quad\quad\quad\quad\vdots\quad\quad\quad\quad\quad\vdots\quad\quad\quad\quad\quad\vdots
$$
$$
 h_{n-1}(x,y)=h_{n-1}(x)-h_{n-2}(x)e_1(y)+\dots+(-1)^ne_{n-1}(y)
$$
$$
\Delta=\frac{e_n(y)}{e_m(x)}
$$
 So we see that   $h_{1}(x,y),\,h_{2}(x,y),\dots,h_{n-1}(x,y),\Delta$ can be expressed by low-triangular matrix in terms of   $e_1(y),\,e_2(y),\dots,e_{n}(y)$ with units and $e_{m}(x)^{-1}$ (up to sign) on the main diagonal. Therefore there exists an automorphism   $\sigma$ of the ring  
 $$
 \Bbb Z[x_1^{\pm1},\dots,x_m^{\pm1},y^{\pm 1}_1,\dots,y^{\pm 1}_n] ^{S_m\times S_m}$$ such that 
$$ 
  \varphi(\sigma_i(y))=h_{i}(x,y),\,i=1,\dots,n-1,\, \varphi(\sigma_{n}(y))=\Delta,
  $$ 
  $$ \varphi (\sigma_i(x))=\sigma_i(x),\, i=1,\dots,m.
 $$
 Therefore  by Lemma \ref{alt1} we see that
$$
H_{\lambda_{1},\dots,\lambda_{m}}h_{1}^{j_{1}}h_{2}^{j_2}\dots h_{n-1}^{j_{n-1}}\Delta^{j_n}
$$
also form a basis of the kernel. And the Theorem follows from the inductive assumption.

\end{proof}

\begin{definition}  Let  $m,n$ be two nonnegative integers.  Let also  $u_1,u_2,\dots,$ $\,\,v_1,v_2, \dots$ be two infinite sets of variables and  $t$ is an additional variable. We assume that $u_0=1,\,v_0=1$ and $u_i=0,\, v_i=0$ for  $i<0$. Set  $w_i=u_i-tv_{-i-m+n},\, i\in\Bbb Z$ and denote by  $ U^{\pm}_{m,n}$   the ring generated by  $u_1,u_2,\dots,$ $\, v_1,v_2, \dots,\,t,$ and relations
$$
R_{i_1,\dots,i_{m+1}}(w)=0,\,\, \text{for any}\,\, (i_1,\dots,i_{m+1})\in \Bbb Z^{m+1}
$$

\end{definition}

We want to show that  $\Lambda^{\pm}_{m,n}$ is isomorphic to  $ U^{\pm}_{m,n}$.
Let us prove first that  $t$ is invertible in the ring $ U^{\pm}_{m,n}$.

\begin{lemma}  Element  $t$  is invertible   in the ring $ U^{\pm}_{m,n}$. 
\end{lemma}
\begin{proof}
Let us consider the subring $\frak A\subset  U^{\pm}_{m,n}$  generated by    $u_i,\,v_i,i\ge1$.  Let us prove that element  $t$ satisfies algebraic equation with coefficients in  $\frak A$  and  that the constant term is equal to $1$. 

For that consider a relation
$$
 \left|\begin{array}{cccc}
   w_{0}& w_{1}& \ldots & w_{m}\\
\vdots&\vdots&\ddots&\vdots\\
  w_{-m}& w_{-m+1}& \ldots & w_{0}\\
  \end{array}\right|=0
$$
The left hand side of this relation is a polynomial in $t$.  If we substitute  in this relation  $t=0$ we see that the constant term is $1$. Therefore the equation has a form 
$$
a_kt^k+a_{k-1}t^{k-1}+\dots +a_1t+1=0
$$
or
$$
t(a_kt^{k-1}+a_{k-1}t^{k-2}+\dots +a_1)+1=0
$$
therefore $t$ is invertible.
\end{proof}

Now we want to construct some set of linear generators of the ring  $ U^{\pm}_{m,n}$.
\begin{definition} Let  $I=(i_1,\dots,i_p), J=(j_1,\dots,j_q)$ such that $(I,J)\in X^{\pm}(m,n)$. Set 
$$
R(I,J)= R_{I}(w) u_1^{j_1}u_2^{j_2}\dots u_{q-1}^{j_{q-1}}t^{j_q}
$$
\end{definition}

\begin{thm}\label{right}  Let $(I,J)\in X^{\pm}(m,n)$  then elements   $R(I,J)$ linearly generate the ring  $U^{\pm}_{m,n}$.

\end{thm}
\begin{proof}
Since  determinant changes its sign after transposition of two rows we can suppose that members of the sequence  $I$ does not necessary strictly decrease. We will use induction on  $mn$. Let  $mn=0$. Then either   $m=0$, or $n=0$. In the first case  $\hat l( I)=0, \hat l(J)=n$. Therefore the relations  come to the form $w_i=0,i\in \Bbb Z$  and they are equivalent to the relations $u_i=0,\,v_i=0$, if  $i>n$ and  $1-tv_n=0,\,u_n-t=0$,\, $u_i-tv_{n-i}=0,\,1\le i\le n-1$. Therefore  the ring  $U^{\pm}(0,n)$ is generated algebraically  by  $u_1,\dots,u_{n-1}, t,t^{-1}$  and Theorem is true in this case.

In the second case $\hat l(I)=m, \,\hat l( J)=0$ and we need to show that  $U^{\pm}(m,0)$ is a linear span of the elements  $R_I,\,\hat l( I)=m$. But this follows from the Theorem \ref{relations1}. 

Suppose now that  $mn>0$.  We have  $m-n=(m-1)-(n-1)$,  therefore from the defining relations of the ring $U^{\pm}_{m,n}$ it follows that there exists a homomorphism 
$$
\psi : U^{\pm}_{n,m}\rightarrow U^{\pm}_{m-1,n-1}
$$
which sends generators to generators. Again from the definition of the algebra $U^{\pm}_{n,m}$ it follows that  the kernel of this homomorphism is the ideal generated by  $R_{I}(w),\,l( I)=m$.  Therefore in order to prove the Theorem it is enough to prove  that product $R_I(w)u_j$ for 
 $j\ne1,\dots, n-1$ can be expressed as a linear combination some of $R(\tilde I,\tilde J),\, (\tilde I,\tilde J)\in X^{\pm}(m,n)$. And we also need to prove the same for  $R_I(w)v_j$, for any $j$.
 
 Let us consider the first case.   We will prove this statement induction on $j-n+1$. If  $j=n-1$, then it is clear. Let  $j>n-1$. Consider relation
$R{i_1,\dots,i_m,j-m}(w)=0$.  If we expand the determinant along its last row we get 
$$
R_I(w)w_j+\sum_{i\in\{j-m,\dots,j-1\}}(-1)^{i+1}R_{\tilde I}^{(i)}(w) w_{i}=0
$$
Since for   $i\in\{j-m,\dots,j-1\}$ we have  $i\ge j-m \ge n-m$, then $w_i=u_i,\,u_i-t$ \,(the last case is possible if  $j=n$) and by induction and Lemma  \ref{identity}  $R_I(w)w_j$ is a linear combination some of  $R(I,J)$.

Let us prove the same statement for product  $R_I(w)v_j$ and  $j> 0$ also by induction. If  $j=1$, then  $w_{n-m-1}=u_{n-m-1}-tv_1$ therefore we can replace  $v_1$ by $w_{n-m-1}$. Then we can consider relation  $R_{i_1,\dots,i_m,n-m-1}(w)=0$. Using the same argument as before we get the necessary statement  for  $j=1$. Let $j>1$, then $w_{n-m-j}=u_{n-m-j}-tv_j$ and we can replace   $v_j$ на $w_{n-m-j}$ and can consider relation  $R_{i_1,\dots,i_m,n-m-j}(w)=0$. Using this relation and inductive assumption we prove the Theorem.
\end{proof}
\begin{corollary} Rings  $\Lambda^{\pm}_{m,n}$ and  $U^{\pm}_{m,n}$ are isomorphic.
\end{corollary}
\begin{proof} Let us consider a map $\varphi$
$$
\varphi(u_i)=h_i,\,\, \varphi(v_i)=h^*_{i},\,\,i\ge 1,\, \varphi(t)= (-1)^{n-m}\frac{y_1\dots y_n}{x_1\dots x_m}
$$
By Lemma   \ref{alt1}
$$
\det(H_{\lambda_i-i+j})_{1\le i,j\le m+1}=0
$$
 Therefore the map $\varphi$ can be extended to  the homomorphism 
 $$
\varphi : U_{m,n}\longrightarrow \Lambda^{+y}_{m,n}.
$$
It is easy to check that homomorphism $\varphi$ sends the family of linear generators of the ring $U_{m,n}$  to the basis of the ring  $\Lambda^{+y}_{m,n}$. Therefore this is an isomorphism.

\end{proof}
\vskip0.3cm

\section{Jacobi-Trudy formulae and Euler supercharacters}
It is well known that  Schur polynomials is a natural basis of the ring $\Lambda_{m}$ and  super Schur  polynomials is a natural basis of the ring $\Lambda_{m,n}$. In the case of the rings $\Lambda_{m,n}^{+y}$ and $\Lambda_{m,n}^{\pm}$  there is a natural basis consisting of the supercharacters of irreducible finite dimensional modules as well. But until now a closed explicit formula for them is not known. It is possible to try to use super-analogues  of composite Schur  functions (see \cite{MV}), but they do not generate the whole algebra  $\Lambda_{m,n}^{\pm}$ in general. So we  use Euler supercharacters instead.  There is an explicit formula for them according to Serganova \cite {GS} and they linearly generate the algebra $\Lambda_{m,n}^{\pm}$. Of course there are many families  of Euler supercharacters which form a basis in the algebra $\Lambda_{m,n}^{\pm}$. We chose those which are closely related to the Kac modules and they are a natural generalisation of super Schur polynomials from the Jacobi-Trudy formula point of view. 

First, we prove some technical lemmas.
Let   $a_i,\,b_i,i\ge0$ be two sequences of elements from a commutative algebra  $\frak A$  such that $a_0=b_0=1$. Consider  two formal series 
$
f(t)=\sum_{i\ge0}a_it^i,\quad   g(t)=\sum_{i\ge0}b_it^i\quad $
and suppose that  $f(t)g(t)=1$.

\begin{remark}
   For a partition $\lambda$ we will denote as usual by $l(\lambda)$ the length of $\lambda$ and by $\mid\lambda\mid$ the number $\lambda_1+\lambda_2+\dots $. We can also consider a partition $\lambda$ as a sequence  of  nonnegative integers $\lambda=(\lambda_1,\lambda_2,\dots,)$. In this case the number $\hat l(\lambda)$ is also defined. But this number is always grater or equal to $\l(\lambda)$ and it depends on how many zeros we put at the end of the $\lambda$.  For  example, let $\lambda=(3,3,2,2,1,0,0,0,0)$. Then $l(\lambda)=5$ but $\hat l(\lambda)=9.$
\end{remark}

\begin{lemma}\label{conj1}  Let  $\lambda$  be a partition such that  $l(\lambda)\le p,\,l(\lambda')\le r.$  Then
$$
\det(a_{\lambda_i-i+j})_{1\le i,j\le p}=(-1)^{|\lambda|}\det(b_{\lambda'_i-i+j})_{1\le i,j\le r}
$$
 \end{lemma}
 \begin{proof} See \cite{Ma}. For reader convenience   we reproduce the proof here. Let  us define two matrices $A=(a_{i-j})$ and $B=(b_{i-j})$ (we assume that $a_i=b_i=0$ for $i<0$).  Then the previous condition means that  $AB=1$ and  $\det A=\det B=1$. By the formula for minors   of mutually inverse matrices (see \cite {Gant}) we have 
 $$
 A(I,J)=(-1)^{|I|+|J|}B(\bar J,\Bar I),\quad I,J\subset\{1,\dots,N\},
 $$
 and  $\bar I,\bar J$  are complements to $I,\, J$. Set $J=\{1,\dots,p\},\,\,\,I=\{\lambda_p+1,\lambda_{p-1}+2,\dots,\lambda_1+p\}$. Then $\bar J=\{p+1,\dots,p+r\}$ and according to  \cite{Ma} $\bar I=\{p+j-\lambda'_j\}$. Therefore 
 $$
 A(I,J)=\det(a_{ij})_{i\in I,\,j\in J}= 
 \det(a_{p-i+1,p-j+1})=\det(a_{\lambda_i-i+j})
 $$
 $$
 B(\bar J\bar I)=\det (b_{ji})_{j\in \bar J,i\in\bar I}=\det(b_{p+j-p-i+\lambda'_i})=\det(b_{\lambda'_i-i+j})
 $$
 \end{proof}
We also need a dual form of composite symmetric polynomials.  
\begin{lemma}\label{dual} Let  $\nu$ and  $\mu$ be partitions such $l(\mu)+l(\nu)\le m$. Then the following  equality is valid in the ring $\Lambda_m^{\pm}$

$$
\left|\begin{array}{cccc}
    h^*_{\nu_{q}}&   h^*_{\nu_{q}-1}& \ldots &   h^*_{\nu_{q}-q-p+1}\\
\vdots&\vdots&\ddots&\vdots\\
   h^*_{\nu_{1}+q-1}&  h^*_{\nu_{1}+q-2}& \ldots &  h^*_{\nu_{1}-p}\\
   h_{\mu_{1}-q}& h_{\mu_{1}-q+1}& \ldots & h_{\mu_{1}+p-1}\\
\vdots&\vdots&\ddots&\vdots\\
  h_{\mu_{p}-q-p+1}& h_{\mu_{p}-q-p+2}& \ldots & h_{\mu_{p}}\\
\end{array}\right|=
$$
$$
\left|\begin{array}{cccc}
   e^*_{\nu'_{s}}& e^*_{\nu'_{s}-1}& \ldots & e^*_{\nu'_{s}-s-r+1}\\
\vdots&\vdots&\ddots&\vdots\\
   e^*_{\nu'_{1}+s-1}& e^*_{\nu'_{1}+s-2}& \ldots & e^*_{\nu'_{1}-r}\\
   e_{\mu'_{1}-s}& e_{\mu'_{1}-s+1}& \ldots & e_{\mu'_{1}+r-1}\\
\vdots&\vdots&\ddots&\vdots\\
  e_{\mu'_{r}-s-r+1}& e_{\mu'_{r}-s-r+2}& \ldots & e_{\mu'_{r}}\\
\end{array}\right|
$$

\end{lemma}
\begin{proof} For any  $i$ we have  $e^*_i=e^*_me_{m-i}$. Therefore the determinant  on the right hand side takes a form 
$$
(e^*_m)^s\left|\begin{array}{cccc}
   e_{m-\nu'_{s}}& e_{m-\nu'_{s}+1}& \ldots & e_{m-\nu'_{s}+s+r-1}\\
\vdots&\vdots&\ddots&\vdots\\
   e_{m-\nu'_{1}-s+1}& e_{m-\nu'_{1}-s+2}& \ldots & e_{m-\nu'_{1}+r}\\
   e_{\mu'_{1}-s}& e_{\mu'_{1}-s+1}& \ldots & e_{\mu'_{1}+r-1}\\
\vdots&\vdots&\ddots&\vdots\\
  e_{\mu'_{r}-s-r+1}& e_{\mu'_{r}-s-r+2}& \ldots & e_{\mu'_{r}}\\
\end{array}\right|
$$
According to our assumptions  $m-\nu'_1\ge\mu'_1$, so by Lemma  \ref{conj1}  the previous determinant is equal to
$$
(e^*_m)^s\left|\begin{array}{cccc}
    h_{\lambda_1}&   h_{\lambda_1+1}& \ldots &   h_{\lambda_1+q+p-1}\\
\vdots&\vdots&\ddots&\vdots\\
  h_{\lambda_{p+q}-q-p+1}& h_{\mu_{p+q}-q-p+2}& \ldots & h_{\lambda_{p+q}}\\
\end{array}\right|
$$
where  $\lambda$ is the partition conjugated  to partition $(m-\nu'_s,\dots,m-\nu'_1,\mu'_1,\dots,\mu'_r)$. 
It is easy to check that  $\lambda=(\mu_1+\nu_1,\dots,\mu_p+\nu_1,\nu_1-\nu_q,\dots,\nu_1-\nu_2,0)$.
Therefore by Lemma  \ref{Com}  the right hand side of the equality which we are proving is equal to  $E_{\chi}$, where $\chi=(\mu_1,\dots,\mu_p,0,\dots,0,-\nu_s,\dots,-\nu_1)$. The left hand side is also  equal  to $E_{\chi}$ by the same Lemma.
\end{proof}

 We actually need a generalisation of the Lemma \ref{conj1}. There should be a direct proof the Lemma below. But we will use the previous Lemma instead. 
 Let    $a_i,\,a^*_{i},\,b_i,\,b^*_i$  be four  sequences of elements from a commutative algebra  $\frak A$  such that $a_0=a_0^*=b_0=b^*_0=1$ and $a_i,=a^*_{i}=b_i=b^*_i=0$ for $i<0$. Consider  four formal series 
$$
f(t)=\sum_{i\ge0}a_it^i,\quad  f^*(t)=\sum_{i\le0} a^*_it^i,\quad g(t)=\sum_{i\ge0}b_it^i,\quad  g^*(t)=\sum_{i\le0} b^*_it^i
$$
and suppose  that $\,\,f(t)g(t)=1,\,\, f^*(t) g^*(t)=1$.

\begin{thm} \label{conj3} Let $\nu,\mu$  are partitions such that   $$l(\nu)=q,\,l(\mu)=p, \, l(\nu')=r,\,l(\mu')=s.$$  Then
$$
\left|\begin{array}{cccc}
    a^*_{\nu_{q}}&   a^*_{\nu_{q}-1}& \ldots &   a^*_{\nu_{q}-q-p+1}\\
\vdots&\vdots&\ddots&\vdots\\
    a^*_{\nu_{1}+q-1}&   a^*_{\nu_{1}+q-2}& \ldots &   a^*_{\nu_{1}-p}\\
   a_{\mu_{1}-q}& a_{\mu_{1}-q+1}& \ldots & a_{\mu_{1}+p-1}\\
\vdots&\vdots&\ddots&\vdots\\
  a_{\mu_{p}-q-p+1}& a_{\mu_{p}-q-p+2}& \ldots & a_{\mu_{p}}\\
\end{array}\right|
$$

$$
=(-1)^{|\nu|+|\mu|}\left|\begin{array}{cccc}
   b^*_{\nu'_{s}}&  b^*_{\nu'_{s}-1}& \ldots &  b^*_{\nu'_{s}-s-r+1}\\
\vdots&\vdots&\ddots&\vdots\\
   b^*_{\nu'_{1}+s-1}& \tilde b^*_{\nu'_{1}+s-2}& \ldots &  b^*_{\nu'_{1}-r}\\
   b_{\mu'_{1}-s}& b_{\mu'_{1}-s+1}& \ldots & b_{\mu'_{1}+r-1}\\
\vdots&\vdots&\ddots&\vdots\\
  b_{\mu'_{r}-s-r+1}& b_{\mu'_{r}-s-r+2}& \ldots & b_{\mu'_{r}}\\
\end{array}\right|
$$

\end{thm}
\begin{proof} Let us fix partitions  $\mu$ and $\nu$.  Then choose a natural number $m$  such that  the following conditions are fulfilled  

$1)$ $p+q\le m$

$2)$ any index in the left hand side  of the determinant and any index in the right hand side of the determinant  is not grater then $m$. Equivalent conditions  are $s+q-1\le m$, $r+p-1\le m$.

$3)$ Elements $e^*_i,\, 1\le i\le q+s-1, \, e_i,\, 1\le i\le p+r-1$  algebraically independent. 

Let now  $\frak B\subset \Lambda_{m}^{\pm}$ be the subring generated by  $e^*_i,\, 1\le i\le q+s-1, \, e_i,\, 1\le i\le p+r-1$.  For any natural  $k$ we have 
\begin{equation}\label{elsym}
h_k=\left|\begin{array}{cccc}
    e_{1}&   \ldots &  e_{k-1}&  e_{k}\\
\vdots&\vdots&\ddots&\vdots\\
  0& \ldots &1&  e_{1}\\
\end{array}\right|,\quad h^*_k=\left|\begin{array}{cccc}
    e^*_{1}& 1&  \ldots &   0\\
\vdots&\vdots&\ddots&\vdots\\
  e^*_{k}& e^*_{k-1}&\ldots &  e^*_{1}\\
\end{array}\right|
\end{equation}
therefore the elements of the   determinant   on the left hand side  in the Lemma \ref{dual} belong to  $\frak B$. Consider a homomorphism $\varphi : \frak B\longrightarrow\frak A$ such that
$$
\,\, \varphi(e_i)=(-1)^ib_i,\,1\le i\le q+s-1\,\,\, \varphi(e^*_i)=(-1)^ib^*_i,\, 1\le i\le p+r-1
$$
 according to our assumptions for  $a_i,a^*_i, b_j,b^*_j$ the conditions (\ref{elsym}) are satisfied. Therefore  $\varphi(h_i)=a_i,\, \varphi(h^*_i)=a^*_i$.  If we apply homomorphism  $\varphi$  to the both sides of the equality from Lemma  \ref{dual}  we get the Theorem.
\end{proof}

Now we are ready   to construct  a basis in the ring $\Lambda^{\pm}_{m,n}$. 
\begin{definition} Let 
$(\lambda,\,\mu)$ be two sequences of  non-increasing  integers. Let us write the sequence $\mu$ in the form $\mu=(\tau_1,\dots,\tau_r,0,\dots,0-\nu_s,\dots,-\nu_1)$, where   $\nu,\,\tau$ are partitions. 
   Let $\sigma$ be the sequence defined by the rule 
$$
(\sigma_1,\dots,\sigma_{l+p+k})=(\nu'_1,\dots,\nu'_l,\lambda_1,\dots,\lambda_p,\tau'_1,\dots,\tau'_k)
$$
where  $p=\hat l(\lambda)$ and $'$ means the conjugate partition.
Let us define  an element $K_{\lambda,\mu}$ of the ring  $\Lambda^{\pm}_{m,n}$ by the  formula 
$$
K_{\lambda,\mu}=\det(a_{ij}),\,\,\text{where}\,\,
a_{ij}=\begin{cases}h^*_{\sigma_{i}+i-j}, \,\, 1\le i\le l\\
 H_{\sigma_i-i+j},\, l<i\le l+p\\ 
 h_{\sigma_{i}-i+j},\, l+p< i\le l+p+k\end{cases},
$$
and in all cases $1\le j\le l+p+k$
\end{definition}

Let us denote by $P(n,m)$ the set of pairs of sequences of non-increasing  integers  $(\lambda,\mu)$  and such that 
$$
\hat l(\lambda)\le m,\,\hat l(\mu) \le n,\,\, \hat l(\lambda)-\hat l(\mu)=m-n
$$
 
\begin{thm}  Let  $(\lambda,\mu)\in P(m,n)$ then set of all $
K_{\lambda,\mu}\,
$
form a basis in the ring  $\Lambda^{\pm}_{m,n}$.
\end{thm}
\begin{proof} We will use induction on $mn$. If  $mn=0$, then either  $m=0$, or $n=0$. If   $m=0$, then $\hat l(\lambda)=0,\,\hat l(\mu)= n,\,h_i(\emptyset, y)=(-1)^ie_i(y), \, h^*_i=(-1)^ie^*_i,\,1\le i\le n$. From the Theorem \ref{conj3} and  Lemma  \ref{Com} it follows that 
$$
K_{\emptyset,\mu}(y)=(-1)^{|\nu|+|\tau|}E_{\mu}(y).
$$

  This polynomials form a basis in $\Lambda^{\pm}_n$ again by Lemma \ref{Com}.
  If  $n=0$, then $\mu=\emptyset,\, \hat l(\lambda)=m$  and the statement follows from the Theorem  \ref{alt}.

 Let now $nm>0$. Consider homomorphism
 $$
 \varphi_{m,n} : \Lambda^{\pm}_{m,n}\longrightarrow \Lambda^{\pm}_{m-1,n-1},\quad \varphi(x_m)=\varphi(y_n)
 $$
 and on the other variables it acts identically.
 By inductive assumption polynomials   $\varphi(K_{\lambda,\mu})$ such that 
 $(\lambda,\mu)\in X(m-1,n-1)$ form a basis in  $\Lambda^{\pm}_{m-1,n-1}$. Therefore in order to prove the Theorem we need to show that, polynomials  $K_{\lambda,\mu}$ such that $\hat l(\lambda)=m,\, \hat l(\mu)=n$ form a basis of the kernel of the homomorphism $\varphi_{m,n}$.
 Actually it is enough to prove that  
 \begin{equation}\label{prod}
 K_{\lambda,\mu}=(-1)^{|\tau|+|\nu|}\prod_{j=1}^n\prod_{i=1}^m\left(1-\frac{y_{j}}{x_{i}}\right)E_{\lambda}(x_{1},\dots, x_{m})E_{\mu}(y_1\dots,y_n),\,\,\, 
\end{equation}
This means that  $K_{\lambda,\mu}$ are supercharacters of Kac modules up to a sign.

We have 
$$
K_{\lambda,\mu}\Delta_m(x)\Delta_n(y)=\left\{K_{\lambda,\mu}(x, y)x^{\rho_m}y^{\rho_n}\right\}
$$
where  as before $\{f(x,y)\}$ means  alternation on the group $S_m\times S_n$.
  Applying  alternation operation to the  row number  $l+1$ we come to  equality
$$
K_{\lambda,\mu}\Delta(x)\Delta(y)=\left\{\prod_{j=1}^n\left(1-\frac{y_{j}}{x_{1}}\right)\tilde K_{\lambda,\mu}(x, y)x^{\rho_m}y^{\rho_n}\right\}
$$
where the determinant  $\tilde K_{\lambda,\mu}(x, y)$ differs from the  $K_{\lambda,\mu}$ only in the row number  $l+1$  which is  
$$
(x_1^{\lambda_1-l+1}, x_1^{\lambda_1-l+2},\dots,x_1^{\lambda_1+p+k})
$$
Now let us multiply  every column   (starting from the first one ) of the determinant  $\tilde K_{\lambda,\mu}(x, y)$ by $x_1$ and subtract  it from the following column. Then using the equalities 
$$
h^*_{i-1}-x_1h^*_{i}=-x_1h^*_{i}(x_2,\dots,x_m,y)
$$
$$
h_{i}-x_1h_{i-1}=h_{i}(x_2,\dots,x_m,y)
$$

 and expanding determinant along its row number  $l+1$  we come to the equality 
 
 \begin{equation}\label{ind21}
 K_{\lambda,\mu}\Delta_m(x)\Delta_n(y)=\left\{\prod_{j=1}^n\left(1-\frac{y_{j}}{x_{1}}\right)x_{1}^{\lambda_{1}}K_{\lambda^{(1)},\mu}(x^{(1)}, y)x^{\rho_m}y^{\rho_n}\right\}
\end{equation}
where
$$
K_{\lambda^{(1)},\mu}( x^{(1)}, y)=K_{\lambda_{2},\dots,\lambda_{m},\mu}( x_{2},\dots x_{m}, y_{1},\dots, y_{n})
$$
 If we apply previous arguments   $m$  times then we come to the equality
$$
K_{\lambda,\mu}\Delta_m(x)\Delta_n(y)=\left\{\prod_{j=1}^n\prod_{i=1}^m\left(1-\frac{y_{j}}{x_{i}}\right)x_{1}^{\lambda_{1}}\dots x_{m}^{\lambda_{m}}K_{\emptyset,\mu}(\emptyset, y)x^{\rho_m}y^{\rho_n}\right\}
$$
But we have already proved that 
$$
K_{\emptyset,\mu}( \emptyset, y)=(-1)^{|\mu|+|\nu|}E_{\mu}(y)
$$
And Theorem is proved.
\end{proof}
\begin{corollary} Let  $(\lambda,\mu)\in P(m,n)$ and $\mu$ is a partition  then set of all $K_{\lambda,\mu}$
form a basis in the ring  $\Lambda^{+y}_{m,n}$.
\end{corollary}
\begin{proof} It is clear that if $\mu$ is a partition then $K_{\lambda,\mu}\in \Lambda^{+y}_{m,n}$. So it is enough to prove that such elements linearly generated the ring $\Lambda^{+y}_{m,n}$.  Let us prove it induction on $mn$. If $m=0$, then we already proved that $K_{\emptyset,\mu}(y)=(-1)^{|\mu|}E_{\mu}(y)$ and therefore linearly generate $\Lambda_n$. If  $n=0$, then as before the statement follows from the Theorem  \ref{alt}. If  $mn>0$, then considering homomorphism   $\varphi_{m,n}$  and applying  inductive assumption we see that it is enough to prove that the kernel of the homomorphism $\varphi_{m,n}$  is a linear span of the elements  $K_{\lambda,\nu}$ such that  $\hat l(\lambda)=m,\hat l(\mu)=n$. But in this case $K_{\lambda,\nu}$  can be written in the form  (\ref{prod})  and therefore they linearly generate the kernel.
\end{proof}
\begin{corollary} Let  $(\lambda,\mu)\in P(m,n)$ and $\lambda,\mu$ are  partitions  then set of all $
K_{\lambda,\mu}\,
$
form a basis in the ring  $\Lambda_{m,n}$.
\end{corollary}

\begin{proof} The proof can be given exactly in the same manner as the proof of the previous corollary.\end{proof}

\begin{remark} It is not difficult to show that there is a bijection between set of partitions  $(\lambda,\mu)\in P(m,n)$ and the set of partitions  $\nu$ such that  $\nu_{m+1}\le n$ and the definition of  $K_{\lambda,\mu}$ coincides with the Jacobi-Trudy formula for supersymmetric Schur functions. So in particular  our formulae are natural generalisations of Jacobi-Trudy formulae. 
\end{remark}

 Now let us show that our   $K_{\lambda,\mu}$ are particular case of Euler supercharacters formulae.

Let us define for  $(\lambda,\mu)\in P(m,n)$  two sets of natural numbers 
$$
D_+=[1,p]\times [1,n],\quad D_{-}=[p+1,m]\times [1,q]
$$
where $p=\hat l(\lambda),\,q=\hat l(\mu)$. Set also
$
 x^{\rho_m}=x_1^{m-1}\dots x_m^{0},\,\, y^{\rho_n}=y_1^{n-1}\dots y_n^{0}.
$
\begin{proposition} The following equality is valid
$$
K_{\lambda,\mu}\Delta(x)\Delta(y)=
$$
$$
=(-1)^{a}\left\{\prod_{(i,j)\in D_{+}}\left(1-\frac{y_j}{x_i}\right)\prod_{(i,j)\in D_{-}}\left(1-\frac{x_i}{y_j}\right)x_1^{\lambda_1}\dots x_p^{\lambda_p}y_1^{\mu_1}\dots y_q^{\mu_q}x^{\rho_m}y^{\rho_n}\right\}
$$
where $a=|\tau|+|\nu|$.
\end{proposition} 
\begin{proof} Let us use induction on $p$.   We start with minimal value of $p$.   If  $n\ge m$, then minimal value of $p$ is $0$. If   $n<m$,  the minimal value of  $p$ is $m-n$. In the first case  $q=n-m$. By Lemma  (\ref{conj3}) 
$$
K_{\emptyset,\mu}(x,y)=(-1)^{|\nu|+|\tau|}K_{\mu,\emptyset}(y,x)
$$.
 The determinant on the right hand side contains  $h^*_i(y,x)$ and  $h_i(y,x)$.  The minimal index in the first case  is  $\nu'_s-s-r+1$, in the second case is  $\tau'_r-r-s+1$.  It follows from the definition that 
 $h_i(y,x)=H_i(y,x) $, и $ h^*_i(y,x)=(-1)^{n+1}\frac{x_1\dots x_m}{y_1\dots n}H_{m-n-i}(y,x)$, if $ i>m-n$. By our assumptions  $s+r\le q=n-m$,  therefore both of minimal indexes strictly grate then  $m-n$ so all the small letters   $h,h^*$  can be replaced by the  capital letters $H$.
  Therefore by Lemma  \ref{alt1}  the statement of the Theorem is true in this case.
 In the second case  $m>n, p=m-n,\,q=0$ and statement of the Theorem follows from the Lemma  \ref{alt1}.

 Let  $p>0$, then  $m>0$  so,  applying equality  (\ref{ind21}) and taking into account   that $\tau_1+\nu_1\le p+n-m=(p-1)+n-(m-1)$  we can reduce the proof to the case  $p-1$ and use inductive  assumption.

\end{proof}
\section{Example} Let us consider the most simple non-trivial example when $m=n=1$. In this case
 $$
 \Lambda^{\pm}_{1,1}=\{f\in \Bbb Z[x^{\pm1},y^{\pm1}\mid x\frac{\partial f}{\partial x}+y\frac{\partial f}{\partial y}\in (x-y)\}
$$
and
$$P(1,1)=\{(\lambda,\mu),\lambda,\mu\in\Bbb Z\}\cup\{\emptyset\}
$$
Let us write down the corresponding elements of the basis.

If $\mu=r>0$, then 
\begin{equation}\label{basis+}
K_{\lambda,\mu}=\left|\begin{array}{cccc}
H_{\lambda}&H_{\lambda+1}&\ldots&H_{\lambda+r}\\
    h_{0}&   h_1 & \ldots&  h_{r}\\
\vdots&\vdots&\ddots&\vdots\\
  h_{1-r}& \ldots &h_0&  h_{1}\\
  \end{array}\right|=-\left(1-\frac yx\right)x^{\lambda}y^{\mu}
\end{equation}

If $\mu=-s<0$, then 
\begin{equation}\label{basis-}
K_{\lambda,\mu}=\left|\begin{array}{cccc}
 h^*_{1}&   h^*_0 & \ldots&  h^*_{1-s}\\
\vdots&\vdots&\ddots&\vdots\\
  h^*_{s}& h^*_{s-1} &\ldots &  h^*_{0}\\
  H_{\lambda-s}&H_{\lambda+1-s}&\ldots&H_{\lambda}\\
  \end{array}\right|=-\left(1-\frac yx\right)x^{\lambda}y^{\mu}
\end{equation}
and if $\mu=0$, then 
$
K_{\lambda,0}=H_{\lambda}=\left(1-\frac yx\right)x^{\lambda}
$ and $K_{\emptyset}=1$. 

The ring $\Lambda^{\pm}_{1,1}$ is isomorphic to the ring $U^{\pm}_{1,1}$ generated by $u_i,v_i,t,\,\,i\in\Bbb Z$ such that $u_0=v_0=1,\,\, u_i=v_i=0,\,i<0$  subject to the  relations
\begin{equation}\label{relations11}
\left|\begin{array}{cc}
w_i&w_{i+1}\\
w_j&w_{j+1}\\
  \end{array}\right|=0,\quad w_i=u_i-tv_i,\quad i,j\in\Bbb Z
\end{equation}
and the corresponding homomorphism has the form
$$
\varphi: U^{\pm}_{1,1}\longrightarrow \Lambda_{1,1}^{\pm},\,\,\varphi(u_i)=h_i,\,\varphi(v_i)=h^{*}_i,\,\varphi(t)=\frac{y}{x}
$$

In the case $\Lambda_{1,1}^{+y}$ the basis forms $K_{\lambda,\mu}$ such that $\mu\ge 0$ and $K_{\emptyset}$. 

The ring $\Lambda^{+y}_{1,1}$ is isomorphic to the ring $U^{+}_{1,1}$ generated by $u_i,v_i, i\in\Bbb Z$ such that $u_0=1,\,\, u_i=v_i=0,\,i<0$  subject to the  relations
\begin{equation}\label{relations12}
\left|\begin{array}{cc}
w_i&w_{i+1}\\
w_j&w_{j+1}\\
  \end{array}\right|=0,\quad w_i=u_i-v_i,\quad i,j\in\Bbb Z
\end{equation}
and the corresponding homomorphism has the form
$$
\varphi: U^{\pm}_{1,1}\longrightarrow \Lambda_{1,1}^{+},\,\,\varphi(u_i)=h_i,\,\varphi(v_i)=\frac{y}{x}h^{*}_i,\,\varphi(t)=\frac{y}{x}
$$

In the case $\Lambda_{1,1}^{+y}$ the basis forms $K_{\lambda,\mu}$ such that $\lambda,\mu\ge 0$ and $K_{\emptyset}$. 

The ring $\Lambda_{1,1}$ is isomorphic to the ring $U_{1,1}$ generated by $u_i\in\Bbb Z$ such that $u_0=1,\,\, u_i=0,\,i<0$  subject to the  relations
\begin{equation}\label{relations12}
\left|\begin{array}{cc}
u_i&u_{i+1}\\
u_j&u_{j+1}\\
  \end{array}\right|=0,\quad i\in\Bbb Z_{>0}
\end{equation}
and the corresponding homomorphism has the form
$$
\varphi: U_{1,1}\longrightarrow \Lambda_{1,1},\,\,\varphi(u_i)=h_i.
$$

\section{Acnowledgements}
I am grateful to A.I. Molev and M.L. Nazarov for useful discussions.

 This work has been   funded by the Russian Academic Excellence Project '5-100' and by the Russian Ministry of Education and  Science (grant 1.492.2016/1.4).

\end{document}